\documentclass[10pt,a4paper,reqno]{amsart}

\usepackage{amsfonts,amssymb,amsmath,amsthm,epsfig,euscript,latexsym}
\usepackage{epic, eepic, graphics}
\usepackage{graphicx}

\newcommand{\sym}{\mathcal{S}}

\newcommand{\bob}[1]{m_{#1}}
\newcommand{\rank}{\ell} % rank (but not in the normal poset terminology)
\newcommand{\srank}{\minmax} % special rank
\newcommand{\myadd}{\Phi}

\newcommand{\addi}{\textsf{Add1}}
\newcommand{\addii}{\textsf{Add2}}
\newcommand{\addiii}{\textsf{Add3}}

\newcommand{\MM}{\mathcal{M}}
\newcommand{\NN}{{\mathbb{N}}} % Natural numbers
\newcommand{\R}{\mathcal{R}}      % Our set of permutations
\newcommand{\T}{\mathcal{T}}      % Our set of permutations
\newcommand{\Matchings}{\mathsf{Match}}
\newcommand{\Z}{\Matchings}

\newcommand{\Phip}{\Upsilon}
\newcommand{\Psip}{\Omega}

\newcommand{\twomatrix}{\left(\begin{array}{cc}}
\newcommand{\threematrix}{\left(\begin{array}{ccc}}
\newcommand{\fourmatrix}{\left(\begin{array}{cccc}}
\newcommand{\ematrix}{\end{array}\right)}

\newcommand{\Aseqs}[1]{{\mathsf{Asc}}_{#1}} % Set of ascent sequences
\newcommand{\PAseqs}[2]{{\mathsf{Asc}}_{#1}^{(#2)}} % Set of ascent sequences with runs of length <= #2
\newcommand{\Poset}{\mathsf{P}} % just Poset symbol
\newcommand{\Posetnk}[2]{\Poset_{#1}^{(#2)}} % Set of posets with # of indist elements <= #2

\newcommand{\MC}[2]{\mathsf{M}_{#1}^{(#2)}}
\newcommand{\M}[1]{\mathsf{M}_{#1}}
\newcommand{\ourmap}{\zeta}
\newcommand{\ourindex}{\mathsf{index}}
\newcommand{\ourvalue}{\mathsf{value}}
\newcommand{\ouradd}[1]{\phi(#1)}
\newcommand{\bcdkmap}{\mathfrak{B}}
\DeclareMathOperator{\st}{std}
\DeclareMathOperator{\yell}{m}
\DeclareMathOperator{\ourextra}{\mathsf{extra}}

\theoremstyle{plain}
\newtheorem{theorem}{Theorem}
\newtheorem{definition}[theorem]{Definition}

\newtheorem{lemma}[theorem]{Lemma}
\newtheorem{proposition}[theorem]{Proposition}
\newtheorem{corollary}[theorem]{Corollary}
\theoremstyle{definition}
\newtheorem{example}[theorem]{Example}

\newcommand{\ds}{\displaystyle}
\newcommand{\tpt}{$(\mathbf{2+2})$}
\newcommand{\tptp}{$\mathbf{2+2}$}

\DeclareMathOperator{\maxindist}{\mathrm{maxindist}}
\DeclareMathOperator{\rep}{\mathrm{rep}}
\DeclareMathOperator{\adjpairs}{\mathrm{epairs}}
\DeclareMathOperator{\adjdes}{\mathrm{adjdes}}
\DeclareMathOperator{\echords}{\mathrm{echords}}
\DeclareMathOperator{\cruns}{\mathrm{cruns}}
\DeclareMathOperator{\larcs}{\mathrm{larcs}}

\DeclareMathOperator{\asc}{\mathrm{asc}}
\DeclareMathOperator{\minmax}{\mathrm{minmax}}
\DeclareMathOperator{\levels}{\mathrm{levels}}

\DeclareMathOperator{\last}{\mathrm{last}}

\newcommand{\size}[1]{|#1|}

\begin{document}
\title{Enumerating \tpt-free posets by indistinguishable elements}

\author[M.\ Dukes]{Mark Dukes}
\address{Department of Computer and Information Sciences, 
	University of Strathclyde, Glasgow, G1 1XH, UK
}
\email{mark.dukes@cis.strath.ac.uk, einar.steingrimsson@cis.strath.ac.uk}

\author[S.\ Kitaev]{Sergey Kitaev} \address{
  School of Computer Science, Reykjav\'{i}k University, 101
  Reykjav\'{i}k, Iceland, and Department of Computer and Information Sciences,
        University of Strathclyde, Glasgow, G1 1XH, UK} \email{sergey@ru.is}

\thanks{MD, SK \& ES: The work presented here was supported by grant
  no.\ 090038012 from the Icelandic Research Fund.}

\author[J.\ Remmel]{Jeffrey Remmel}
\address{Department of Mathematics,
University of California, San Diego,
La Jolla, CA 92093-0112,  USA}
\email{remmel@math.ucsd.edu}
\thanks{JR: Partially supported by NSF grant DMS 0654060.}
\author[E.\ Steingr\'{\i}msson]{Einar Steingr\'{\i}msson}

\begin{abstract}
A  poset is said to be \tpt-free if it does not contain an
induced subposet that is isomorphic to {\tptp}, the union of two
disjoint 2-element chains.  Two elements in a poset are
indistinguishable if they have the same strict up-set and the same strict
down-set.  Being indistinguishable defines an equivalence relation on
the elements of the poset.  We introduce the statistic $\maxindist$,
the maximum size of a set of indistinguishable elements.  
We show that, under a bijection of Bousquet-M\'elou et
al.\ \cite{BCDK}, indistinguishable elements correspond to letters
that belong to the same run in the so-called ascent sequence
corresponding to the poset. We derive the generating function for the
number of \tpt-free posets with respect to both $\maxindist$ and the
number of different strict down-sets of elements in the poset.
Moreover, we show that \tpt-free posets $P$ with $\maxindist(P)$
at most $k$ 
are in bijection with upper triangular matrices of nonnegative
integers not exceeding $k$, where each row and each column contains a
nonzero entry.  (Here we consider isomorphic posets to be equal.) In
particular, \tpt-free posets $P$ on $n$ elements with $\maxindist(P) =
1$ correspond to upper triangular binary matrices where each row and
column contains a nonzero entry, and whose entries sum to $n$.  We
derive a generating function counting such matrices, which confirms a
conjecture of Jovovic~\cite{J}, and we refine the generating function
to count upper triangular matrices consisting of nonnegative integers
not exceeding $k$ and having a nonzero entry in each row and column.
That refined generating function also enumerates \tpt-free posets
according to $\maxindist$. Finally, we link our enumerative results 
to certain restricted permutations and matrices.
\end{abstract}

\maketitle
\thispagestyle{empty}

\section{Introduction}

This paper continues the study of enumerative properties of three
distinct equinumerous classes of combinatorial objects, namely,
\tpt-free posets (also known as {\it{interval orders}}, see 
Fishburn~\cite{FISH_BOOK}), ascent sequences, and upper triangular matrices with
nonnegative integer entries and where each row and column contains a
nonzero entry.  We build on the work of Bousquet-M\'elou et
al.~\cite{BCDK}, who presented a bijection between \tpt-free posets
and ascent sequences, and that of Dukes and Parviainen \cite{dp}, who
gave a bijection between ascent sequences and upper triangular
matrices with nonnegative integer entries and no rows or columns of
all zeros.

It is important to note that, as in \cite{BCDK}, we consider, and
count, \tpt-free posets up to isomorphism.  That is, we consider two
such posets to be equal if there is an order preserving bijection
between them.  In \cite{BCDK} the isomorphism classes are referred to
as ``unlabeled posets''.

The central result of this paper is the determination of the
generating function for the number of ascent sequences of length $n$
with $k$ pairs of consecutive elements that are equal.  We call an
ascent sequence with no two consecutive equal entries a {\em primitive
  ascent sequence}.  A special case gives the generating function for
the number of primitive ascent sequences.  We show that under the
bijections mentioned above, primitive ascent sequences correspond to
{\em primitive} \tpt-free posets, defined by having no pair of
elements with the same strict down-sets and the same strict up-sets, and
also to upper triangular binary matrices with no rows or columns of
zeros.  This allows us to prove a conjecture of Jovovic \cite{J} which
states that the generating function for the number of upper triangular
binary matrices with no rows or columns of zeros is given by
\begin{equation}\label{con} K(t) = \sum_{n\geq 0}
\prod_{i=1}^n \left( 1 - \frac{1}{(1+t)^i} \right).
\end{equation}

In order to state our results more precisely, we now introduce the
three main classes of combinatorial structures treated in the paper,
namely, ascent sequences, \tpt-free posets, and upper triangular
matrices with nonnegative integer entries and no rows or columns of
all zeros.

\subsection{Ascent sequences}

An \emph{ascent} in an integer sequence $(x_1,\dots , x_i)$, is a $j$
such that $x_j<x_{j+1}$.  The number of ascents in such a sequence $X$
is denoted $\asc(X)$.

A sequence $(x_1,\dots , x_n ) \in \mathbb{N}^n$ is an {\em{ascent
    sequence of length $n$}} if and only if it satisfies $x_1=0$ and
$$ 0\le x_i \le 1+\asc(x_1,\dots , x_{i-1}) $$
whenever $2\leq i \leq n$.

Let $\Aseqs{n}$ be the collection of ascent sequences of length $n$
and let $\Aseqs{}$ be the collection of all ascent sequences,
including the empty ascent sequence.  If $a \in \Aseqs{n}$ then we
will write $|a|=n$.  For example, $(0,1,0,2,3,1,0,0,2)$ is an ascent
sequence in $\Aseqs{9}$.

A {\em run} in an ascent sequence is a maximal subsequence of
consecutive letters that are all equal.  Let $\PAseqs{}{k}$ be the
collection of ascent sequences whose runs have length at most $k$, and
let $\PAseqs{n}{k}$ be those $a \in \PAseqs{}{k}$ that have $|a|=n$.
A {\em primitive ascent sequence} is an ascent sequence with no runs
of length greater than 1.  Thus, $\PAseqs{n}{1}$ is the set of all
primitive ascent sequences.

Given $a=(a_1,\ldots , a_n) \in \Aseqs{n}$, we call a pair
$(a_i,a_{i+1})$ with $a_i=a_{i+1}$ an {\em equal pair} of the
sequence\footnote{This is sometimes called a {\em{level}} in the
  literature on sequences, not to be confused with the definition of
  level in the present paper.}.  We denote the number of equal pairs
in a sequence $a$ by $\adjpairs(a)$.  For example $\adjpairs
(0,0,0,0,0,1,1,2,1,1)=6$ since
$(a_1,a_2)=(a_2,a_3)=(a_3,a_4)=(a_4,a_5)=(0,0)$ and
$(a_6,a_7)=(a_9,a_{10})=(1,1)$.

\subsection{\tpt-free posets}
Recall that we consider two posets to be equal if they are isomorphic.
A poset is said to be \emph{{\tpt}-free} if it does not
contain an induced subposet that is isomorphic to {\tptp}, the union
of two disjoint 2-element chains.  We let $\Poset$ denote the set of
\tpt-free posets (including the empty poset) and $\Poset_n$ the set of
all such posets on $n$ elements.  For $P \in \Poset$, let $\size{P}$
be the number of elements in $P$.

An important characterization (see \cite{FISH_BOOK,FISH_OPER,
  SKANDERA}) says that a poset is \tpt-free if and only if its strict
down-sets can be ordered linearly by inclusion.  
For a poset ${P} = (P,\prec _p)$ and $x \in P$, the strict
down-set of $x$, denoted $D(x)$, is the set of all $y \in P$
such that $y \prec_p x$.  
Clearly, any poset is uniquely specified by listing the collection of
strict down-sets of each element.  The {\em trivial down-set} is the
empty set.  Thus if ${P}$ is a {\tpt}-free poset, we can write $D(P)=
\{D(x):x \in P\}$ as
$$D(P) =(D_0,D_1, \ldots ,D_{k})$$ where $\emptyset = D_0 \subset D_1
\subset \cdots \subset D_{k}$.  We then say that $x \in P$ is at
        {\it{level}} $i$ if $D(x) = D_i=D_i(P)$ and write $\rank(x) =
        i$.
We also define $\levels(P)$ by setting $\levels(P)=k$, where $k$ is
the index of the highest level in $P$.

We
  denote by $L_i(P) = \{\, x\in P : \rank(x)=i \,\}$ the set of all
  elements at level $i$ and we set
$$L(P) = \big(\,L_0(P),L_1(P),\dots,L_{\levels(P)}(P)\,\big).$$ 
Let $\bob{P}$ be a maximal element of $P$
  whose strict down-set is the smallest of the strict down-sets of $P$'s
  elements.  This element may not be unique but all such elements
  belong to the same level. We define $\srank(P)$ by setting
  $\srank(P)=\rank(\bob{P})$.  
The maximal elements of $P$ are $P\backslash D_{\levels(P)}(P)$.
Thus $\minmax(P)=\min\{ \rank(x) ~:~ x \in P\backslash D_{\levels(P)}(P) \}$.

As a counterpart to the strict down-set $D(x)$ of an element $x$ in a poset,
we let $U(x)$ denote the {\em strict up-set} of $x$, that is,
$U(x)=\{\, y : x\prec_p y \,\}$.  
Given $P \in \Poset_n$, we say that two
elements $x,y \in P$ are {\em indistinguishable} if $D(x)=D(y)$ and
$U(x)=U(y)$. 
We write this as $x \sim_P y$ and note that $\sim_P$ is an equivalence relation on $P$. 
Let us define $\maxindist(P)$ to be the size of the largest equivalence class in $P$.

For example, the elements $c$ and $d$ in the poset of
Example~\ref{ex1} below, as well as the
elements $2$ and $3$ in the poset of Example~\ref{ex2}, are
indistinguishable.  We say that a \tpt-free poset is {\em primitive}
if it contains no pair of indistinguishable elements.  

We let
$\Posetnk{n}{k}$ denote the set of all \tpt-free posets $P$ on $n$ elements 
for which $\maxindist(P)$ is at most $k$.
In particular, $\Posetnk{n}{1}$ 
is the set of primitive \tpt-free posets on $n$ elements.  We define
the statistic $\rep(P)$ for $P\in\Poset$ to be the minimum number of
elements that need to be removed to create a primitive poset.  For
example, the value of this statistic is 1 on the posets in both
Examples~\ref{ex1} and~\ref{ex2}.  Note that $\maxindist(P)=1$ if and
only if $\rep(P)=0$.

	\begin{example}\label{ex1}
	Let $P$ be the following {\tpt}-free poset (reproduced
	from~\cite{BCDK} with kind permission of the authors):
	
	$$
	\begin{minipage}{8.6em}
  	\includegraphics[scale=0.85]{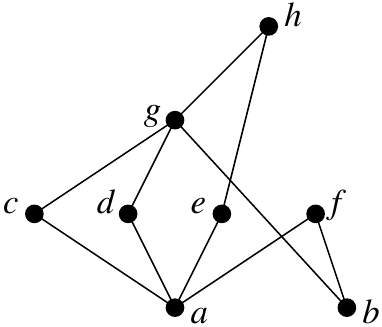}
	\end{minipage} 
	\kern10mm =\kern10mm
	\begin{minipage}{10em}
  	\scalebox{0.70}{\includegraphics{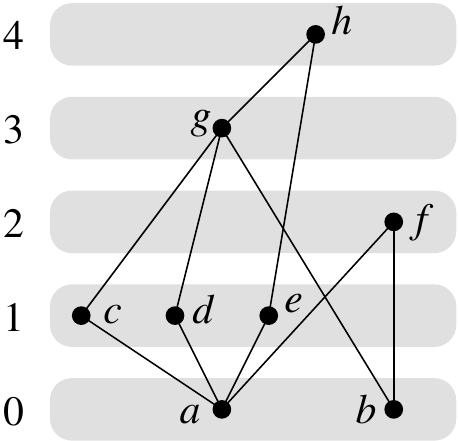}}
	\end{minipage}
	$$ 
	On the right the poset has been redrawn to show the level numbers
	determined by the strict down-set of each element (when compared to the
	strict down-sets of other elements).  
	Notice that $\levels(P)=4$. The strict down-set, level, and strict up-set of each element is as follows:
	$$
	\footnotesize{
	\begin{array}{|c||c|c|c|c|c|c|c|c|} \hline
	x & a & b & c & d & e & f & g & h \\ \hline\hline
	D(x) & \emptyset & \emptyset & \{a\} & \{a\} & \{a\} & \{a,b\} & \{a,b,c,d\} & \{a,b,c,d,e,g\} \\ \hline
	\rank(x) & 0 & 0 & 1&1&1&2&3&4\\ \hline
	U(x) & \{c,d,e,f,g,h\} & \{f,g,h\} & \{g,h\} & \{g,h\} & \{h\} & \emptyset & \{h\} & \emptyset \\ \hline
	\end{array}
	}
	$$
	We therefore have
	${D(a) = D(b)}\subset {D(c) = D(d) = D(e)}\subset{D(f)}\subset {D(g)}\subset D(h)$.
	The strict down-sets for each level are listed along with the elements of each level:
	$$\begin{array}{|c||c|c|c|c|c|} \hline 
	i & 0 & 1 & 2 & 3 & 4 \\ \hline \hline
	D_i(P) & \emptyset & \{a\} & \{a,b\} & \{a,b,c,d\} & \{a,b,c,d,e,g\} \\  \hline
	L_i(P) & \{a,b\} & \{c,d,e\} & \{f\} & \{g\} & \{h\} \\ \hline
	\end{array}
	$$
	\noindent
	The maximal elements of $P$ are $P\backslash D_4(P)=\{f,h\}$.
        Since $D(f) \subset D(h)$ we have $\bob{P}=f$ and
        $\srank(P)=2$.  
	%%% REFEREE \qed
	\end{example}

\subsection{Upper triangular matrices}

Let $\M{n}$ be the set of upper triangular matrices of nonnegative
integers such that
no row or column contains all zero entries, and
the sum of the entries is $n$.
Let $\M{}$ be the set of all such matrices, that is,
$\M{}=\bigcup_{n\ge0}\M{n}$.  For example, $\M{3}$ consists of the
following 5 matrices:
$$(3) ,
        \twomatrix 2 & 0 \\ 0 & 1 \ematrix ,
        \twomatrix 1 & 1 \\ 0 & 1 \ematrix ,
        \twomatrix 1 & 0 \\ 0 & 2 \ematrix ,
        \threematrix 1 & 0 & 0 \\ 0 & 1 & 0 \\ 0 & 0 & 1 \ematrix .
$$

For any $A \in \M{}$, we let $|A|$ be the sum of the entries in $A$,
and we set $\ourextra(A)=|A|-\mathrm{NZ}(A)$, where $\mathrm{NZ}(A)$
is the number of nonzero entries in $A$. Also, let $\ourindex(A)$ be
the smallest value of $i$ such that $A_{i,\dim(A)}$ is nonzero, where
$\dim(A)$ is the number of rows (or columns) in $A$.  As an example,
let $$A=\fourmatrix 1&3&0&0 \\ 0&0&2&0 \\ 0&0&0&5 \\ 0&0&0&2
\ematrix.$$ Then $|A|=1+3+2+5+2=13$, $\mathrm{NZ}(A)=5$,
$\ourextra(A)=13-5=8$, and $\ourindex(A)=3$ since the topmost non-zero
entry in the final column is in the third row.  Let $\MC{n}{k}$ be the
collection of matrices in $\M{n}$ that have no entries exceeding $k$.
In particular, $\MC{n}{1}$ is the set of binary matrices in $\M{n}$.

\subsection{Enumerative results}
Let $p_n$ be the number of \tpt-free posets on $n$ elements.
Bousquet-M\'elou et al.~\cite{BCDK} showed that the generating
function for the number $p_n$ of \tpt-free posets on $n$ elements is

\begin{equation}\label{gf}
P(t)= \sum_{n\geq 0} p_n \, t^n=\sum_{n\ge 0} \prod_{i=1}^{n} \left(
1-(1-t)^i\right).  
\end{equation}

A more general power series $F(t,u,v)$ that takes into account the
statistics {\it{number of levels}} and {\it{level number of the lowest
    maximal element}} is implied by inserting the power series given
in \cite[Proposition 15]{BCDK} into \cite[Lemma 14]{BCDK}.  See
\cite[Section 6]{BCDK} for an overview of these generating functions.
More recently, Kitaev and Remmel \cite{KR} generalized the result of
\cite[Section 6]{BCDK} to derive a generating function that
incorporated two further statistics related to \tpt-free posets.

\subsection{Statements of main results}
In this paper we study the generating function 
\begin{eqnarray*}
G(u,v,y,t) 
&=& \sum_{P \in \Poset} u^{\levels(P)} v^{\minmax(P)} y^{\rep(P)} t^{\size{P}}.
\end{eqnarray*}
Using the bijections of Bousquet-M\'elou et al.~\cite{BCDK} and Dukes
and Parviainen \cite{dp}, respectively, 
this is also the generating function of several statistics 
on ascent sequences and matrices. (This is made clear at the beginning of Section 5.)

We show that $H(u,v,y,t)=G(u,v,y,t)-1$ satisfies the following recurrence:
\begin{equation}\label{Intro1}
H(u,v,y,t)(v-1 -t-tyv +ty+tuv) = t(v-1)- tH(u,1,y,t) + tuv^2
H(uv,1,y,t).
\end{equation}
Using the kernel method, we then show that
\begin{equation}\label{Intro2}
G(u,1,y,t) = 1+\frac{t(1-u)}{\Delta_1} + \sum_{n=1}^{\infty}
  \frac{t(1-u)(1-ty)^n(1+t-ty)^n \prod_{i=1}^n
    \Gamma_i}{\Delta_n\Delta_{n+1}},
\end{equation}
where $\Delta_k =(1-ty)^k(1-u) + u(1+t-ty)^k$ and $ \Gamma_k =
  ({u(1+t-ty)^k})/{\Delta_k}$.
We can then use (\ref{Intro1}) and (\ref{Intro2}) to give
an explicit formula for $G(u,v,y,t)$.

We also show that the generating function for primitive \tpt-free
posets is given by
\begin{equation}\label{introcon} K(t) = \sum_{n\geq 0}
\prod_{i=1}^n \left( 1 - \frac{1}{(1+t)^i} \right)\end{equation} which
confirms a conjecture of Jovovic~\cite{J}.  Primitive \tpt-free posets
are of special interest as one can easily generate all \tpt-free
posets from the primitive ones by specifying the number of copies of
each element.  

Finally, we show that \tpt-free posets for which the $\maxindist$
statistic is at most $k$ correspond to ascent sequences with runs of
length at most $k$, and to upper-triangular matrices with entries not
exceeding $k$.  This allows us to generalize formula~(\ref{introcon})
to prove that
$$\sum_{n\geq 0} |\Posetnk{n}{k}|x^n=\sum_{n\geq 0}
|\MC{n}{k}|x^n=\sum_{n\geq 0} |\PAseqs{n}{k}|x^n=\sum_{n\geq 0}
\prod_{i=1}^n \left(1-\left( \dfrac{1-x}{1-x^k}
  \right)^i\right).$$

\subsection{Outline of the paper} In Section~\ref{decomp} we recall
the bijection of Bousquet-M\'elou et al.~\cite{BCDK} between \tpt-free
posets and ascent sequences.  In Subsection~\ref{enumeration} we show
that $|\PAseqs{n}{k}|=|\Posetnk{n}{k}|=|\MC{n}{k}|$.  In
Section~\ref{conj-solved} we derive the generating function for
primitive ascent sequences and for ascent sequences with runs of
length at most~$k$.  In Section~\ref{enum} we derive our formula for
$G(u,v,y,t)$ and discuss its specialization $G(u,1,0,t)$ corresponding
to primitive ascent sequences.  Finally, in Section~\ref{last_sec}, we
show that restricting ascent sequences by bounding the run-length
corresponds, via the bijection in~\cite{BCDK}, to bounding the length
of a sequence of consecutive descents on the restricted permutations
in~\cite[Sect. 2]{BCDK}.  We also show that a similar statement holds
for the relationship between ascent sequences and Stoimenow matchings.

\section{{\tpt}-free posets, ascent sequences and matrices}\label{decomp}
\subsection{Constructing {\tpt}-free posets from ascent sequences}

In this subsection we review the essential parts of \cite[Section
  3]{BCDK} with proofs omitted.  We describe a bijective map
$\bcdkmap$ from the collection of ascent sequences of length $n$ to
the collection of (canonically labeled) {\tpt}-free posets on $n$
elements. The mapping is a step by step procedure which constructs a poset
element by element.
One always starts with the single poset having one element labeled `1'.
The $j$th element of the poset to be inserted is labeled `$j$'.

Central to the construction are the three addition rules: {\addi},
{\addii} and {\addiii}.  Given a poset $P \in \Poset_m$, and a value
$i \in [0,1+\levels(P)]$, we produce a poset $\myadd(P,i) \in
\Poset_{m+1}$ where the new poset element, regardless of its position,
has label $m+1$.  The appropriate addition rule to use depends on
whether $i \in [0,\minmax(P)]$, $i=1+\levels(P)$ or $i \in
[\minmax(P)+1, \levels(P)]$.

Since a {\tpt}-free poset $P$ is uniquely determined by the pair
$\big(D(P),L(P)\big)$, in defining the addition 
operations below it suffices to only specify how $D(P)$ and
$L(P)$ change.  
Note that {\addi} leaves $\levels(P)$ unchanged, whereas {\addii} and {\addiii} 
increase $\levels(P)$ by one.

Given $P \in \Poset_n$, let us write $D_i=D_i(P)$ and $L_i=L_i(P)$.
Given a value $i$ with $0\leq i \leq 1+\levels(P)$, let $\myadd(P,i)$ be
the poset $Q$ obtained from $P$ in the following way:
\begin{enumerate}
\item[(\addi)] If $0\leq i \leq \srank(P)$ then set $D(Q)=D(P)$ and
$$L(Q)=(L_0,\ldots , L_i\cup \{n+1\} ,\ldots , L_{\levels(P)} ).$$
\item[(\addii)] If $i=1+\levels(P)$ then set
\begin{eqnarray*}
D(Q)&=&(D_0,\ldots , D_{\levels(P)},P) \\
L(Q)&=&(L_1,\ldots , L_{\levels(P)},\{n+1\}).
\end{eqnarray*}
\item[(\addiii)] If $\srank(P)<i< 1+\levels(P)$ then set
\begin{eqnarray*}
\MM &=& L_0 \cup \cdots \cup L_{i-1}\backslash D_{\levels(P)} \\
D(Q)&=& (D_0,\ldots , D_i, D_i \cup \MM,\ldots , D_{\levels(P)} \cup \MM)\\
L(Q)&=& (L_0,\ldots , L_{i-1},\{n+1\},L_i,\ldots ,L_{\levels(P)}).
\end{eqnarray*}
\end{enumerate}

An important property of the above addition operations is that
$\srank(\myadd(P,i))$ $=$ $i$, since all maximal elements below level $i$
are covered and therefore not maximal in $\myadd(P,i)$.  
Note that the single poset $P \in P^{(1)}$ is such that
$D(P)=(\emptyset)$, $L(P)=(\{1\})$ and $\levels(P)=0$.

\begin{definition}
Given $x=(x_1,\ldots , x_n) \in \Aseqs{n}$, 
let $\bcdkmap(x)=P^{(n)}$ where
$P^{(m)} := \myadd(P^{(m-1)},x_{m})$ for all $1<m\leq n$.
\end{definition}

\begin{example}\label{example}\label{ex2}
In this example we construct the poset $P=\bcdkmap(x)$ where $x=(0,1,1,0,2,0,1)\in\Aseqs{7}$.
We begin from the poset $P^{(1)}$ with just a single element, and successively construct $P^{(2)},\ldots ,P^{(7)}=P$
according to the addition rules.
	The poset $P^{(1)}$ is the poset with one element labeled '1'.
        This element is the only element at level 0 of $P^{(1)}$,
        illustrated as follows:
\ \\[1.5em]
\centerline{  	\scalebox{0.70}{\includegraphics{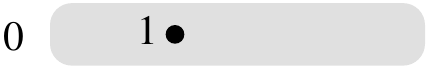}}}
\ \\[1em]
\begin{minipage}{23em}
	Since $1=1+\levels(P^{(1)})$, the poset $P^{(2)}=\myadd(P^{(1)},1)$ is constructed by applying rule {\addii}.
	The new element is labeled `2':
\end{minipage}
$\hspace*{1em}$
\begin{minipage}{11.5em}
	$$\mapsto\quad
	\begin{minipage}{8.5em}
  	\scalebox{0.70}{\includegraphics{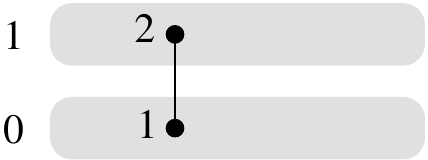}}
	\end{minipage}
	$$ 
\end{minipage}
\ \\[1.5em]
\begin{minipage}{23em}
	Since $1\in [0,\minmax(P^{(2)})]$, the poset $P^{(3)}=\myadd(P^{(2)},1)$ is constructed by applying {\addi}:
\end{minipage}
$\hspace*{1em}$
\begin{minipage}{11.5em}
	$$ 
	\mapsto\quad
	\begin{minipage}{8.5em}
  	\scalebox{0.70}{\includegraphics{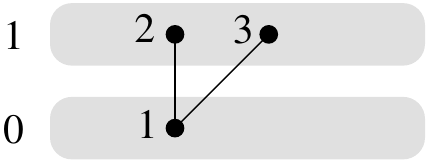}}
	\end{minipage}
	$$ 
\end{minipage}
\ \\[1.5em]
\begin{minipage}{23em}
	The poset $P^{(4)}=\myadd(P^{(3)},0)$ is constructed by applying {\addi}:
\end{minipage}
$\hspace*{1em}$
\begin{minipage}{11.5em}
	$$
	\mapsto \quad
	\begin{minipage}{8.5em}
  	\scalebox{0.70}{\includegraphics{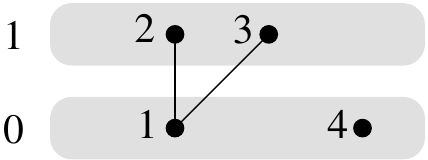}}
	\end{minipage}
	$$ 
\end{minipage}
\ \\[1.5em]
\begin{minipage}{23em}
	The poset $P^{(5)}=\myadd(P^{(4)},2)$ is constructed by applying {\addii}:
\end{minipage}
$\hspace*{1em}$
\begin{minipage}{11.5em}
	$$\mapsto\quad
	\begin{minipage}{8.5em}
  	\scalebox{0.70}{\includegraphics{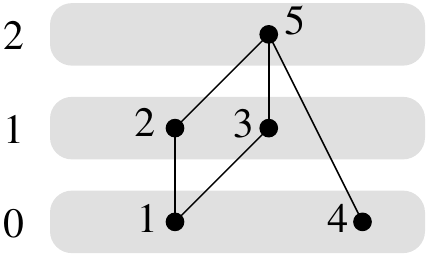}}
	\end{minipage}
	$$ 
\end{minipage}
\ \\[1.5em]
\begin{minipage}{23em}
	The poset $P^{(6)}=\myadd(P^{(5)},0)$ is constructed by applying {\addi}:
\end{minipage}
$\hspace*{1em}$
\begin{minipage}{11.5em}
	$$\mapsto\quad
	\begin{minipage}{8.5em}
  	\scalebox{0.70}{\includegraphics{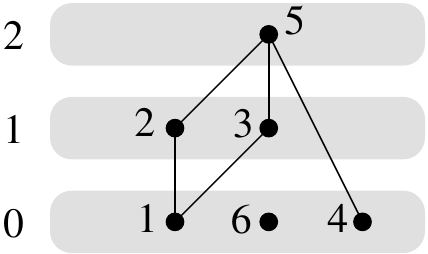}}
	\end{minipage}
	$$ 
\end{minipage}
\ \\[1.5em]
\begin{minipage}{12em}
	The poset $P^{(7)}=\myadd(P^{(6)},1)$ is constructed by applying {\addiii}.
	Note that we introduce a new empty level between levels $x_6-1$ and $x_6$
	and insert a new single element with the same downset as the elements
	that were on that level. Then all the elements above it have the set 
	$\MM=\{6\}$ included in their downsets.
\end{minipage}
$\hspace*{1em}$
\begin{minipage}{24.5em}
	$$\mapsto\;
	\begin{minipage}{10em}\scalebox{0.70}
  	{\includegraphics{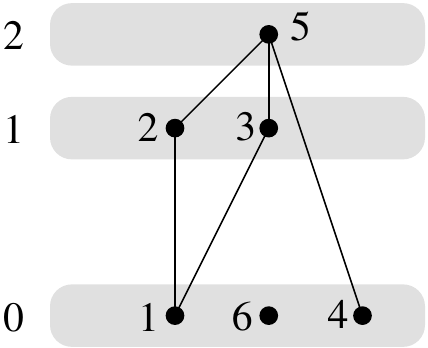}}
	\end{minipage} \mapsto\;
	\begin{minipage}{10em}
  	\scalebox{0.70}{\includegraphics{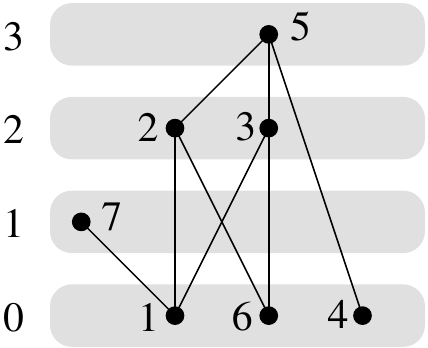}}
	\end{minipage}
	$$
\end{minipage}
\ \\[1.5em]
\begin{minipage}{21em}
\begin{minipage}{9em}
  \scalebox{0.70}{\includegraphics{secon2}}
\end{minipage}
=\quad
\begin{minipage}{8.6em}
  \includegraphics[scale=0.85]{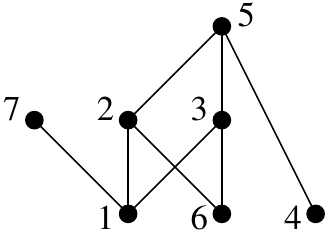}
\end{minipage} 
\end{minipage}
$\hspace*{0.5em}$
\begin{minipage}{14em}
Finally $P=P^{(7)}$ and we have the canonically labeled poset $\bcdkmap(x)\in \Poset_7$.
\end{minipage}
%%% REFEREE \qed
\ \\
\end{example}

\subsection{Bounded run lengths in ascent sequences}\label{enumeration}

In this subsection we prove Propositions~\ref{lemma} and~\ref{thm:mnk}
establishing the relation between runs in ascent sequences,
indistinguishable elements in \tpt-free posets, and entries of
restricted upper-triangular matrices.  To be more precise, we show
that $$|\PAseqs{n}{k}|=|\Posetnk{n}{k}|=|\MC{n}{k}|.$$

We use the following proposition to deal with ascent sequences in
order to obtain results for posets.

\begin{proposition}
\label{lemma} Let $x \in \Aseqs{n}$ and $P \in \Poset_n$ with $P=\bcdkmap(x)$.
Given $i<j$ we have that
$i \sim_P j$ if and only if $x_{i}=x_{i+1}=\cdots = x_{j}$.
\end{proposition}

\begin{proof}
We first show that $i\sim_{P} (i+1)$ iff $x_{i}=x_{i+1}$. Let
$x_{i}=x_{i+1}$. Think of $P$ being created by adding elements one by
one and using the rules (Add1)--(Add3) and assume that $i+1$ has just
entered $P$ ($i$ has already been added to $P$ on the previous
step). Since $x_{i}=x_{i+1}$, at this point
$(D(i),U(i))=(D(i+1),U(i+1))$ where $U(i)=U(i+1)=\emptyset$.
Moreover, from the definitions of (Add1)--(Add3), $D(i)$ and $D(i+1)$
will either both stay unchanged or will be changing in the same way
while adding extra elements to $P$.  The same applies to $U(i)$ and
$U(i+1)$.  Thus, $i\sim_{P} (i+1)$. On the other hand, if $x_{i}\neq
x_{i+1}$, then $D(i)\neq D(i+1)$ after adding $i+1$ to $P$ ($i$ and
$(i+1)$ will be on different levels) and the definitions of
(Add1)--(Add3) guarantee that $i$ and $(i+1)$ will remain on different
levels while adding extra elements to $P$. That is, $i\not\sim_{P}
(i+1)$.

Next we show that if $x_i=x_j$ and there exists $x_k\neq x_i$ such
that $i<k<j$ then $i\not\sim_{P} j$. To prove this we need the notion
of the {\em modified ascent sequence} $\hat{x}$ and its properties
introduced in~\cite[Section 4]{BCDK}. If $x_ix_{i+1}\ldots x_j$
contains an element $x_s>x_i$ then we can take the minimum such $s$ to
see that $s\in U(i)$ but $s\not\in U(j)$ showing that $i\not\sim_{P}
j$. Otherwise, there must exist an ascent $x_sx_{s+1}$ with
$x_{s+1}\leq x_i$ and $i<s<j$. This would mean that in
$\hat{x}=\hat{x}_1\hat{x}_2\ldots\hat{x}_n$, we have
$\hat{x}_i>\hat{x}_j$, so $i$ will be on a higher level than $j$ in
$P$ and $i\not\sim_{P} j$.

To complete the proof we show that if $i\not\sim_{P} j$ then either
$x_i\neq x_j$ or $x_i=x_j$ but there exists $x_k\neq x_i$ such that
$i<k<j$. This, however, is a direct corollary to the definition and
properties of the modified ascent sequence $\hat{x}$ whose maximal
runs of equal elements correspond to the level distribution of
elements in $P$.  Namely, two different runs of the same element in
$\hat{x}$ correspond to elements in $P$ with the same down-sets but
with different up-sets --- this is a fact that is not explicitly
mentioned in ~\cite[Section 4]{BCDK} but it can be proved.
\end{proof}

We have the following immediate corollary to
Proposition~\ref{lemma}.

\begin{corollary} Primitive \tpt-free posets on $n$ elements are 
in one-to-one correspondence with primitive ascent sequences of 
length $n$.
\end{corollary}

One more corollary follows from the proof of Proposition~\ref{lemma}.

\begin{corollary}\label{pairs-indist} The statistic $\rep$ on
$\Posetnk{n}{k}$ corresponds to the statistic $\adjpairs$ on
$\PAseqs{n}{k}$ under $\bcdkmap$.
\end{corollary}

\subsection{Restricted matrices and ascent sequences}
In Dukes and Parviainen ~\cite{dp} a bijection $\Gamma: \M{n} \to
\Aseqs{n}$ was presented.  Here we find it convenient to describe the
inverse $\ourmap:\Aseqs{n} \to \M{n}$ of this map.  Given $A \in
\M{n}$, let $\ourindex(A)$ be the smallest value of $i$ such that
$A_{i,\dim(A)}$ is nonzero.  Given a value $m$ such that $0\leq m \leq
\dim (A)$, we define the matrix $\ouradd{A,m}$ according to the
following:
\begin{enumerate}
\item[(i)] If $0\leq m< \ourindex(A)$ then let $\ouradd{A,m}$ be the
  matrix $A$ with the entry at position $(m+1,\dim(A))$ increased by
  1.
\item[(ii)] If $\ourindex(A) \leq m \leq \dim(A)$ then we form
  $\ouradd{A,m}$ in the following way.  Let $A'$ be the matrix with
  $\dim(A')=\dim(A)+1$ formed by inserting a row of zeros immediately
  after row $m$ of $A$, and a column of zeros immediately after column
  $m$ of $A$.  Let $A'_{m+1,\dim(A)+1}=1$.  Swap the values $A'_{i,m+1}$
  and $A'_{i,\dim(A)+1}$ for all $1\leq i \leq m$.  Call the resulting
  matrix $\ouradd{A,m}$.\\
\end{enumerate}
As an example of the second construction, let
$$A=\threematrix 1 & 7 & 1 \\ 0 & 9 & 3 \\ 0 & 0 & 2 \ematrix .$$
Then $\ouradd{A,2}$ is given by
$$
\fourmatrix 1 & 7 &\phantom{0} & 1 \\ 0 & 9 & & 3 \\ \phantom{0} &&& \\ 0 & 0 & & 2 \ematrix
\to
\fourmatrix 1 & 7 &0 & 1 \\ 0 & 9 &0 & 3 \\ 0&0&0&1 \\ 0 & 0 & 0 & 2 \ematrix
\to
\fourmatrix 1 & 7 & 1&0 \\ 0 & 9 & 3&0 \\ 0&0&0&1 \\ 0 & 0 & 0 & 2 \ematrix = \ouradd{A,2}.
$$

Given $x=(x_1,\ldots , x_n) \in \Aseqs{n}$, let $\epsilon$ be the
empty matrix.  Define $\ouradd{\epsilon,0}:=(1)$ and
$$\ourmap (x) =
\ouradd{\cdots\ouradd{\ouradd{\epsilon,x_1},x_2}\cdots,x_n}.$$

\begin{proposition} \label{thm:mnk}
 For each $n\geq 0$ and $k\geq 1$, we have $\ourmap(x) \in \MC{n}{k}$
 if and only if $x \in \PAseqs{n}{k}$.
\end{proposition}

\begin{proof}
Let $x=(x_1,\ldots , x_n) \in \Aseqs{n}$.  Define
$A^{(i)}=\ourmap(x_1,\ldots , x_i)$ for all $1\leq i \leq n$.  Let us
suppose that $x \not \in \PAseqs{n}{k}$ so that there exists $i$ such
that $x_i=x_{i+1}=\cdots = x_{i+k}=c$.  Since
$A^{(i)}=\ouradd{A^{(i-1)},x_i}$, we have
$\ourindex(A^{(i)})-1=x_i=c$.  Let $d=\dim(A^{(i)})$.  So the entry
$A^{(i)}_{c+1,d}\geq 1$.  Since $A^{(i+1)}=\ouradd{A^{(i)},c}$, and
$x_{i+1}=c < \ourindex(A^{(i)})=1+c$, the rule (i) is used and we have
$A^{(i+1)}$ as the matrix $A^{(i)}$ with the entry at position
$(c+1,d)$ increased by 1.  So $A^{(i+1)}_{c+1,d}\geq 2$.  Doing this
repeatedly, we find that $A^{(i+k)}_{c+1,d} \geq 1+k$, which means
that $A^{(i+k)} \not \in \PAseqs{i+k}{k}$, and so $A^{(n)} \not \in
\MC{n}{k}$ since neither of the construction rules (i) or (ii) can
decrease an entry of a matrix (although entries may be moved).

Next we prove that $\ourmap(x) \not \in \MC{n}{k}$ $\Rightarrow$ $x
\not\in \PAseqs{n}{k}$.
The inverse of $\ourmap$ was recursively described in \cite{dp}.  In
order to find the ascent sequence $(x_1,\ldots,x_n)$ corresponding
to $A \in \M{n}$, one finds that there is a unique $f(A) \in
\M{n-1}$ and value $x_n=\ourindex(A)-1$ such that
$A=\ouradd{f(A),x_n}$ and $f(A)=\ourmap(x_1,\ldots,x_{n-1})$.  To
determine the reduced matrix $f(A)$ one must invoke one of the three
removal rules, called $\mathsf{Rem1}-\mathsf{Rem3}$ in ~\cite{dp}.
We present the argument without describing these rules explicitly.

Let $X=\ourmap(x) \in \M{n}\backslash \MC{n}{k}$.  Then there is at
least one entry $X_{ab}$ in $X$ with $X_{ab} \geq k+1$.  At some stage
during the deconstruction process, the value $X_{ab}$ will be in the
rightmost column of $f(f(\cdots f(A)\cdots))$.  If there are non
negative values above it, they will be removed in due course of the
deconstruction.  One then has a matrix $B \in \M{m}$, where
$\ourvalue(B)=X_{ab}$ and
$$A= \ouradd{ \cdots \ouradd{\ouradd{B,x_{m+1}},x_{m+2}} \cdots ,
  x_n}.$$ Since $X_{ab}\geq k+1$, the next $k$ removals will invoke
${\mathsf{Rem1}}$, thereby giving $x_{m-1}=x_{m-2}=\ldots =
x_{m-k}$.  Since $\ourvalue(B) \geq 1$, regardless of which removal
rule is used next, one finds that $x_{m-k-1}=x_{m-k}$.  This implies
there are at least $k+1$ consecutive entries in the ascent sequence
which take the same value.  Hence $x \not \in \PAseqs{n}{k}$.
\end{proof}

\section{Enumerating ascent sequences with restricted runs}\label{conj-solved}
The primitive ascent sequences of length $n$ are in one-to-one
correspondence with matrices in $\MC{n}{1}$, see ~\cite[Thm.  13]{dp}.
Jovovic~\cite{J} conjectured the generating function (\ref{con}) for
the number of matrices in $\MC{n}{1}$ (see ~\cite[A138265]{oeis}).
Here we prove this conjecture (Theorem~\ref{count}) by using the
bijective correspondence with ascent sequences, and we also generalize
the generating function (\ref{con}) to count more complicated objects
(Theorem~\ref{bcor}).

In Bousquet-M\'elou et al.~\cite{BCDK} it was shown that
\begin{equation} \label{markone}
P(x)= \sum_{a \in \Aseqs{}} x^{|a|} = \sum_{n\geq 0} \prod_{i=1}^n
\left(1 - \left(1-x\right)^i\right).
\end{equation}
Let $K(x) = \displaystyle\sum_{n\geq 0} k_n x^n$ where
$k_n=|\PAseqs{n}{1}|$ is the number of primitive ascent sequences of
length $n$.  Due to Propositions~\ref{lemma} and~\ref{thm:mnk}
we have
$$K(x) =\sum_{n\geq 0} |\MC{n}{1}|x^n=\sum_{n\geq 0}
|\Posetnk{n}{1}|x^n.$$

We now give an explicit formula for $K(x)$, proving a conjecture of
Jovovic~\cite{J}.

\begin{theorem}\label{count}  We have
$$K(x) =\ds \sum_{n\geq 0} \prod_{i=0}^n \left( 1 - \frac{1}{(1+x)^i}
\right).$$ \end{theorem}

\begin{proof} Every ascent sequence $a=(a_1,\ldots , a_n)$ may be written
uniquely in the form
$$(b_1^{m_1},\ldots , b_k^{m_k})$$ where $(b_1,\ldots , b_k)$ is a
primitive ascent sequence, and $m_i$ is the number of consecutive
entries of $b_i$ in $a$.  For example, if
$a=(0,0,1,1,1,0,2,2,3,1,1,0,4)$ then
$a=(0^2,1^3,0^1,2^2,3^1,1^2,0^1,4^1)$ and $b=(0,1,0,2,3,1,0,4)$ is the
underlying primitive ascent sequence with multiplicities
$(2,3,1,2,1,2,1,1)$.  A primitive ascent sequence of length $n\geq 1$
gives rise to an infinite number of ascent sequences by choosing
multiplicities $(m_1,\ldots , m_n) \in \mathbb{N}^n$.  Therefore,
\begin{equation} \label{marktwo}
P(t) \;=\; \sum_{n\geq 0} k_n (t+t^2+\cdots )^n \;=\; \sum_{n\geq 0}
k_n \left(\dfrac{t}{1-t}\right)^n =K\left( \frac{t}{1-t}\right).
\end{equation}
Setting $x = {t}/({1-t})$, we see that $t = {x}/({1+x})$ so that

\begin{equation}\label{mark3}
K(x) = P\left(\frac{x}{1+x}\right) = \sum_{n\geq 0} \prod_{i=1}^n
\left( 1 - \frac{1}{(1+x)^i} \right).
\end{equation}
\end{proof}

Let 
\[ 
\ds B_k(x)= \sum_{n\geq 0}|\PAseqs{n}{k}|x^n=\sum_{n\geq
  0}|\MC{n}{k}|x^n=\sum_{n\geq 0} |\Posetnk{n}{k}|x^n,
\] 
where the latter two identities were established in
Theorem~\ref{thm:mnk} and Proposition~\ref{lemma}.  Then we have
the following theorem which generalizes Theorem~\ref{count} (the case
$k=1$) and gives the generating function for the number of ascent
sequences that have a run of length at most $k$.

\begin{theorem}\label{bcor} We have
$$\ds \sum_{n\geq 0}|\PAseqs{n}{k}|x^n=\sum_{n\geq 0} 
|\MC{n}{k}|x^n=\sum_{n\geq 0} |\Posetnk{n}{k}|x^n=\sum_{n\geq 0}
\prod_{i=1}^n \left(1-\left( \dfrac{1-x}{1-x^{k+1}}
  \right)^i\right).$$
\end{theorem}
\begin{proof}
It is easy to see that
\begin{equation*}
B_k(x)=\sum_{n\geq 0} k_n (x+x^2+\cdots + x^k)^n =
K\left(\frac{x(x^k-1)}{(x-1)}\right) = \sum_{n\geq 0} \prod_{i=1}^n
\left(1-\left( \dfrac{1-x}{1-x^{k+1}}
  \right)^i\right).
\end{equation*}
\end{proof}

\section{Enumeration of ascent sequences by ascents, equal pairs,
and last letter}\label{enum}

The theorems in this section concern the enumeration of ascent
sequences.
Let 
\begin{eqnarray}
G(u,v,y,t) &=& \sum_{s \in \Aseqs{}} u^{\asc(s)} v^{\last(s)}
y^{\adjpairs(s)} t^{|s|} 
 = \sum_{a,\yell,b,n\geq 0} G_{a,\yell,b,n} u^a v^{\yell} y^b t^n
\end{eqnarray}
be the generating function for ascent sequences according to 
the statistics introduced in Section 1.
The value $G_{a,b,\yell,n}$ is the number of ascent sequences of
length $n$ with $a$ ascents, $b$ equal pairs, and last letter $\yell$.

From the correspondences in ~\cite{BCDK,dp}  and Corollary~\ref{pairs-indist}, we see
that this generating function is also the generating function of {\tpt}-free posets
and our upper-triangular matrices:
\begin{eqnarray}
G(u,v,y,t)
&=& \sum_{P \in \Poset} u^{\levels(P)} v^{\minmax(P)} y^{\rep(P)}
t^{\size{P}}\label{eqa}\\
&=& \sum_{A \in \M{}} u^{\dim(A)-1} v^{\ourindex(A)-1} y^{\ourextra(A)} t^{|A|}. \label{eqb}
\end{eqnarray}

Let $H(u,v,y,t)=G(u,v,y,t)-1$ be the generating function for these
statistics over all nonempty ascent sequences.

\begin{lemma}\label{lem:Grec}
The formal power series $H(u,v,y,t)$ satisfies
\begin{eqnarray}\label{eq:G1}
&&H(u,v,y,t)(v-1 -t-tyv +ty+tuv) \nonumber \\
&&= t(v-1)- tH(u,1,y,t) + tuv^2 H(uv,1,y,t).
\end{eqnarray}
\end{lemma}

\begin{proof}
It is easy to see that 
\begin{eqnarray*}
\lefteqn{G(u,v,y,t)}\\ &=& 
1+t+ t\sum_{n \geq 1,\atop a,b,\yell \geq 0}
	G_{a,b,\yell,n} t^n\left(\left(\sum_{i=0}^{\yell - 1} u^a v^i y^b\right) + u^a
	v^{\yell} y^{b+1} + \sum_{i=\yell+1}^{a+1} u^{a+1} v^i y^b\right) \\ 
&=& 1+t+
	t\sum_{n \geq 1,\atop a,b,\yell \geq 0} G_{a,b,\yell,n} t^n u^a y^b
	\left(\frac{v^{\yell} -1}{v-1} +yv^{\yell} +u \frac{v^{a+2}
  -v^{\yell+1}}{v-1}\right) \\ 
&=& 1+t + t(G(u,v,y,t)-1) \left( \frac{1+y(v-1) -uv}{v-1} \right) 
	- \frac{t}{v-1}(G(u,1,y,t)-1)\\ && + \frac{tuv^2}{v-1}
(G(uv,1,y,t)-1).
\end{eqnarray*}
Since $G(u,v,y,t)=1+H(u,v,y,t)$ we find that 
\begin{eqnarray*}
{(v-1) H(u,v,y,t)} &=& t(v-1) + H(u,v,y,t) (t+tyv -ty-tuv)\\ && -
tH(u,1,y,t) + tuv^2 H(uv,1,y,t).
\end{eqnarray*}
\end{proof}

We use the above lemma to give an expression for the power series $G(u,1,y,t)$.

\begin{theorem}\label{thm:G1}
We have
$$
G(u,1,y,t) = 1+\frac{t(1-u)}{\Delta_1} + \sum_{n=1}^{\infty}
  \frac{t(1-u)(1-ty)^n(1+t-ty)^n \prod_{i=1}^n
    \Gamma_i}{\Delta_n\Delta_{n+1}},
$$
where $\Delta_k =(1-ty)^k(1-u) + u(1+t-ty)^k$ and $ \Gamma_k =
  ({u(1+t-ty)^k})/{\Delta_k}.$
\end{theorem}

\begin{proof}
The left hand side of the functional equation (\ref{eq:G1}) vanishes
when the coefficient to $H(u,v,y,t)$ is zero.  This happens precisely
when $v$ is
\begin{equation}\label{eq:v}
W(u,y,t)= \frac{1+t-ty}{1+tu-ty}.
\end{equation}
Replacing $v$ by $W(u,y,t)$ in (\ref{eq:G1}) gives
$$0 = \frac{t^2(1-u)}{1+tu-ty} -tH(u,1,y,t) +
tu\left(\frac{1+t-ty}{1+tu-ty}\right)^2
H\left(u\frac{1+t-ty}{1+tu-ty},1,y,t\right)$$ and hence
\begin{equation}\label{eq:G2}
H(u,1,y,t) = \frac{t(1-u)}{1+tu-ty} +
u\left(\frac{1+t-ty}{1+tu-ty}\right)^2
H\left(u\frac{1+t-ty}{1+tu-ty},1,y,t\right).
\end{equation}

Next let
\begin{equation*}
\Delta_k =(1-ty)^k(1-u) + u(1+t-ty)^k.
\end{equation*}
 It is easy to check that $\Delta_1 = 1+tu-ty$.  Also let
\begin{equation*}
\Gamma_k = \frac{u(1+t-ty)^k}{\Delta_k}.
\end{equation*}
The following identities are immediate:
\begin{equation*}
(1-u)|_{u=\Gamma_s} ~=~ \frac{\Delta_s}{\Delta_s} -
  \frac{u(1+t-ty)^s}{\Delta_s} ~=~ \frac{(1-ty)^s(1-u)}{\Delta_s},
\end{equation*}
\begin{equation*}
\Delta_k|_{u=\Gamma_s} ~=~ \frac{(1-ty)^k(1-ty)^s(1-u)}{\Delta_s}+
\frac{u(1+t-ty)^s(1+t-ty)^k}{\Delta_s} ~=~
\frac{\Delta_{k+s}}{\Delta_s},
\end{equation*}
\begin{equation*}
\frac{(1-u)}{\Delta_k}|_{u=\Gamma_s} ~=~ \frac{\Delta_{s}}{\Delta_{k+s}}
\frac{(1-ty)^s(1-u)}{\Delta_s} ~=~ \frac{(1-ty)^s(1-u)}{\Delta_{k+s}},
\end{equation*}
\begin{equation*}
\Gamma_k|_{u=\Gamma_s} ~=~ (1+t-ty)^k\frac{u(1+t-ty)^s}{\Delta_{s}}
\frac{\Delta_{s}}{\Delta_{s+k}} ~=~ \Gamma_{s+k}.
\end{equation*}
We can then rewrite (\ref{eq:G2}) as
\begin{equation}\label{eq:G3}
H(u,1,y,t) =
\frac{t(1-u)}{\Delta_1}+\frac{(1+t-ty)}{\Delta_1}\Gamma_1 H(\Gamma_1,1,y,t).
\end{equation}
Iterating (\ref{eq:G3}) gives 
\begin{eqnarray}\label{eq:G4}
H(u,1,y,t)
 &=& \frac{t(1-u)}{\Delta_1} + 
\frac{(1+t-ty)}{\Delta_1}\Gamma_1\left\{
t\frac{(1-ty)(1-u)}{\Delta_1} \frac{\Delta_1}{\Delta_2}  \right.  \nonumber \\ && \left.  
+ (1+t-ty)
\frac{\Delta_1}{\Delta_2} \Gamma_2
G(\Gamma_2,1,y,t)\right\} \nonumber \\
&=& \frac{t(1-u)}{\Delta_1} + \frac{t(1-u)(1-ty)(1+t-ty) \Gamma_1}{\Delta_1 \Delta_2} + \nonumber \\
&& \frac{(1+t-ty)^2 \Gamma_1 \Gamma_2}{\Delta_2} H(\Gamma_2,1,y,t).
\end{eqnarray}
If we iterate (\ref{eq:G4}), then we find that
\begin{eqnarray}\label{eq:G5}
H(u,1,y,t) &=& \frac{t(1-u)}{\Delta_1} + \frac{t(1-u)(1-ty)(1+t-ty)
  \Gamma_1}{\Delta_1 \Delta_2} \nonumber \\ &&
+\frac{t(1-u)(1-ty)^2(1+t-ty)^2 \Gamma_1\Gamma_2}{\Delta_2 \Delta_3}
\nonumber \\ &&+ \frac{t(1-u)(1-ty)^3(1+t-ty)^3
  \Gamma_1\Gamma_2\Gamma_3}{\Delta_3 \Delta_4} \nonumber
\\ &&+\frac{(1+t-ty)^4 \Gamma_1\Gamma_2\Gamma_3
  \Gamma_4}{\Delta_4}H(\Gamma_4,1,y,t).
\end{eqnarray}
One can then easily prove by induction that
\begin{eqnarray}
H(u,1,y,t) &=& \frac{t(1-u)}{\Delta_1} + \sum_{n=1}^{2^n-1}
\frac{t(1-u)(1-ty)^n(1+t-ty)^n \prod_{i=1}^n
  \Gamma_i}{\Delta_n\Delta_{n+1}} + \nonumber \\ &&
\frac{(1+t-ty)^{2^n}\prod_{i=1}^{2^n} \Gamma_i}{\Delta_{2^n}}
H(\Gamma_{2^n},1,y,t).
\end{eqnarray}
Since each $\Gamma_i$ has a factor of $u$, it is easy to see that, as
a formal power series in $u$,
\begin{equation}\label{finalG(u,1,y,t)}
H(u,1,y,t) = \frac{t(1-u)}{\Delta_1} + \sum_{n=1}^{\infty}
\frac{t(1-u)(1-ty)^n(1+t-ty)^n \prod_{i=1}^n
  \Gamma_i}{\Delta_n\Delta_{n+1}}.
\end{equation}
\end{proof}
The first few terms of $G(u,1,y,t)$ are
\begin{equation}
G(u,1,y,t) = 1+\frac{P_0(t,y)}{(1-ty)} + \frac{P_1(t,y)}{(1-ty)^3}u +
\frac{P_2(t,y)}{(1-ty)^6}u^2 + \frac{P_3(t,y)}{(1-ty)^{10}}u^3 +
O(u^4)
\end{equation}
where the power series $P_i(t,y)$ are given in Figure~\ref{ppolys}.
\begin{figure}
$$
\begin{array}{|l|c|} \hline
i & P_i(t,y) \\ \hline \hline 0 & t\\ \hline 1 &
t^2(1-ty)+t^3\\ \hline 2 & t^3+4 t^4+4 t^5+t^6-3 t^4 y-8 t^5 y-4 t^6
y+3 t^5 y^2+4 t^6 y^2-t^6 y^3 \\ \hline 3 & t^4+11 t^5+33 t^6+42
t^7+26 t^8+8 t^9+t^{10}-6 t^5 y-55 t^6 y \\ & -132 t^7 y-126 t^8 y-52
t^9 y-8 t^{10} y+15 t^6 y^2+110 t^7 y^2+198 t^8 y^2 \\ & +126 t^9
y^2+26 t^{10} y^2-20 t^7 y^3-110 t^8 y^3-132 t^9 y^3-42 t^{10} y^3
\\ & +15 t^8 y^4+55 t^9 y^4+33 t^{10} y^4-6 t^9 y^5-11 t^{10}
y^5+t^{10} y^6 \\ \hline
\end{array}
$$
\caption{The first four power series $P_i(t,y)$.}
\label{ppolys}
\end{figure}
For example, for the ascent sequences with a single ascent one can see
that
\begin{eqnarray*}
\frac{P_1(t,y)}{(1-ty)^3} &=& \frac{t^2(1-ty)+t^3}{(1-ty)^3} \;=\;
\frac{t^2}{(1-ty)^2} + \frac{t^3}{(1-ty)^3}\\ &=& \sum_{n \geq 2}
(n-1)y^{n-2}t^n + \sum_{n \geq 3} \binom{n-1}{2}y^{n-3}t^n.
\end{eqnarray*}
Here the first sum accounts for ascent sequences of the form $0^a1^b$
where $a,b \geq 1$ and the second sum accounts for ascent sequences of
the form $0^a1^b0^c$ where $a,b,c \geq 1$.

We can now use Lemma \ref{lem:Grec} and Theorem \ref{thm:G1} to give 
an expression for $G(u,v,y,t)$.  That is, if we define 
$\Delta_0 =1$, then by  Theorem \ref{thm:G1}, we have that
\begin{equation}%\label{J1}
G(u,1,y,t) = 1+\sum_{n\geq 0}
\frac{t(1-u)(1-ty)^n(1+t-ty)^n \prod_{i=1}^n \Gamma_i}{\Delta_n \Delta_{n+1}}
\end{equation}
and
\begin{equation}
G(uv,1,y,t) = 1+\sum_{n\geq 0}
\frac{t(1-uv)(1-ty)^n(1+t-ty)^n \prod_{i=1}^n \bar{\Gamma}_i}{\bar{\Delta}_n
\bar{\Delta}_{n+1}}
\end{equation}
where $\bar{\Delta}_0 =1$ and
$\bar{\Delta}_k = (1-ty)^k(1-uv)+uv(1+t-ty)^k$ and
$\bar{\Gamma}_k = (uv(1+t-ty)^k)/{\bar{\Delta}_k}$ for
$k \geq 1$.
Thus we have the following theorem.
\begin{theorem}\label{main}
\begin{eqnarray*}
\lefteqn{G(u,v,y,t)}\\
&=& 1+\frac{t}{(v-1 -t-tyv +ty+tuv)}\bigg( v-1\\
&& - t \sum_{n\geq 0} (1-ty)^n (1+t-ty)^n \left\{  
	\dfrac{(1-u)\prod_{i=1}^n \Gamma_i}{\Delta_n \Delta_{n+1}}
	-
	\dfrac{uv^2(1-uv)\prod_{i=1}^n \bar{\Gamma}_i}{\bar{\Delta}_n \bar{\Delta}_{n+1}}
 \right\}\bigg) .
\end{eqnarray*}
\end{theorem}

The first few terms of this power series are as follows:
\begin{eqnarray*}
G(u,v,y,t) &=& 1+t+(u v+y) t^2+\left(u+u^2 v^2+2 u v y+y^2\right)
t^3\\ &&+\left(u^2+2 u^2 v+u^2 v^2+u^3 v^3+3 u y+3 u^2 v^2 y+3 u v
y^2+y^3\right) t^4 \\
&&+O(t^5). 
\end{eqnarray*}

\subsection{Enumeration of primitive ascent sequences by ascents.}
Primitive ascent sequences, that is, ascent sequences with no 2-runs,
correspond to setting $y=0$ in $G(u,1,y,t)$.
When $y=0$, the expression $\Delta_k$ becomes
$
(1-u) + u(1+t)^k
$
and $\Gamma_k$ becomes
$
{u(1+t)^k}/{\delta_k}.
$
Thus we have the following;
\begin{corollary} Let $\delta_k = (1-u) + u(1+t)^k$ and $\gamma_k = u(1+t)^k/\delta_k$.
Then
\begin{equation}\label{finalG(u,1,0,t)}
G(u,1,0,t) = 1+ \frac{t(1-u)}{\delta_1} + \sum_{n=1}^{\infty}
\frac{t(1-u)(1+t)^n \prod_{i=1}^n \gamma_i}{\delta_n\delta_{n+1}}.
\end{equation}
\end{corollary}
Unfortunately we cannot derive a generating function for the number
of primitive ascent sequences from $G(u,1,0,t)$ by setting $u=1$
(this generating function is derived in Section~\ref{conj-solved}
using different arguments).  The power series for the
first few terms in the expansion of $G(u,1,0,t)$ (about $u=0$),
\begin{equation}
G(u,1,0,t) = 1+\sum_{n=0}^4 q_n(t)u^n + O(u^5)
\end{equation}
are given in Figure~\ref{qpolys}.
\begin{figure}
$$\begin{array}{|l|c|} \hline i & q_i(t) \\ \hline \hline 0 & t
    \\ \hline 1 & t^2+t^3 \\ \hline 2 & t^3+4 t^4+4 t^5+t^6 \\ \hline
    3 & t^4+11 t^5+33 t^6+42 t^7+26 t^8+8 t^9+t^{10} \\ \hline 4 &
    t^5+26 t^6+171 t^7+507 t^8+840 t^9+865 t^{10}+584 t^{11}+262
    t^{12}+76 t^{13} \\ & +13 t^{14}+t^{15} \\ \hline
 5 & t^6+57
    t^7+718 t^8+4017 t^9+12866 t^{10}+26831 t^{11}+39268 t^{12} \\ & +
    42211 t^{13} +34221 t^{14}+21184 t^{15}+10015 t^{16}+3571
    t^{17}+933 t^{18}\\ & +169 t^{19}+19 t^{20}+t^{21} \\ \hline 
\end{array}
$$
\caption{The first six power series $q_i(t)$.}
\label{qpolys}
\end{figure}
Note that the power series $q_n(t)$ are unimodal for $0\leq n \leq 7$.
It would be nice to have a combinatorial proof of this for general $n$.

\section{Permutations and matchings corresponding to ascent sequences
  with bounded run-length}\label{last_sec}

We conclude by mentioning the restricted permutations and matchings
that correspond, via the maps in \cite{BCDK,dp}, to ascent sequences
with bounded run-length.  First,
we recall a few definitions from the papers \cite{BCDK,cdk}.

Let $V=\{v_1,v_2,\dots,v_n\}$ with $v_1<v_2<\dots<v_n$ be any finite
subset of $\NN$.  The \emph{standardization} of a permutation $\pi$ of
the elements of $V$ is the permutation $\st(\pi)$ of $\{1,\ldots ,n\}$
obtained from $\pi$ by replacing the letter $v_i$ with the letter $i$.
As an example, $\st(39685) = 15342$.  Let
\begin{equation*}
\R_{n} = \{\,\pi_1\dots\pi_n \in \sym_n : \text{ if
$\st(\pi_i\pi_j\pi_k)=231$ then $j\neq i+1$ or $\pi_i\neq \pi_k+1$}
\,\},
\end{equation*}
where $\sym_n$ is the set of permutations of $\{1,2,\ldots,n\}$.

In other words, $\R_n$ is the set of permutations of $[n]$ where, in
each occurrence of the pattern 231, either the letters corresponding
to the 2 and the 3 are nonadjacent, or else the letters corresponding
to the 2 and the 1 are not adjacent in value.  For instance, the
occurrence 463 in $\pi=546123$ violates both conditions, since 4 and 6
are adjacent letters in $\pi$ and 4 and 3 are adjacent values.  Note
that both $\R_n$ and $\T_n$ are defined in terms of avoidance of
\emph{bivincular patterns}, which were defined in~\cite{BCDK}.

Also, let $\T_n$ be the subset of $\R_n$ whose permutations have no
adjacent letters that are adjacent in value and in decreasing order,
that is, no descent consisting of letters that differ in size by one.
In the permutation $546123$, mentioned above, there is one violation
of that condition, namely the 54.

Let $\R_{n}^{(k)}$ be the subset of permutations $\pi\in\R_n$ such
that there do not exist integers $i$ and $m$ with $\pi_i=m$,
$\pi_{i+1}=m-1$, $\ldots$, $\pi_{i+k}=m-k$.  Also, for any general
pattern $p$, let $p(\pi)$ be the number of occurrences of $p$ in
$\pi$.

A \emph{matching} of the set $[2n]=\{1,2,\ldots,2n\}$ is a partition
of $[2n]$ into subsets of size 2, each of which is called an
\emph{arc}.  The smaller number in an arc is its {\em opener} and the
larger one its {\em closer}.  A matching is said to be
\emph{Stoimenow} if it has no pair of arcs $\{a,b\}$ and $\{c,d\}$,
with $a<b$ and $c<d$, satisfying one (or both) of the following
conditions:
\begin{enumerate}
\item $a=c+1$ and $b<d$,
\item $a<c$ and $b=d+1$.
\end{enumerate}
In other words, a Stoimenow matching has no pair of arcs such that one
is nested within the other and the openers, or closers, of the two
arcs differ by 1.

Let $\Matchings_n$ denote the set of Stoimenow matchings on $[2n]$ and
$\Matchings$ the set of all such matchings.
If $(i,j)$ and $(i+1,j+1)$ are arcs in a matching $M$, we say that
they are {\emph{similar}}.  Let $\echords(M)$ be the minimum number of
arcs in $M$ one has to remove to obtain a matching without similar
arcs.  Let $\Z_n^{(k)}$ be the collection of matchings
$M\in\Matchings_n$ such that for no pair $i$ and $j$ do all of
$(i,j),(i+1,j+1),\ldots , (i+k,j+k)$ belong to $M$.

Bijections $\Lambda: \R_n \to \Aseqs{n}$ and $\Psi': \Matchings_n \to \Aseqs{n}$
were presented in \cite[Thm. 1]{BCDK} and \cite[Thm. 7]{cdk}, respectively.
Let us write $\Phip$ and $\Psip$ for their respective inverses, so that we
have $\Phip: \Aseqs{n} \to \R_{n}$ and $\Psip: \Aseqs{n} \to \Matchings_n$.
It is then fairly easy to prove the following theorem and corollary,
and we omit these proofs.

\begin{theorem}\label{match}
We have 
\begin{enumerate}
\item[(i)]
$\Phip (\PAseqs{n}{k}) = \R_{n}^{(k)}$, and 
\item[(ii)] $\Psip (\PAseqs{n}{k}) = \Z_n^{(k)}$.  
\end{enumerate}
In particular, 
\begin{enumerate}
\item[(iii)] $\Phip (\PAseqs{n}{1})
= \T_{n}$, and
\item[(iv)] $\Psip (\PAseqs{n}{1}) =  \{ M \in \Matchings_n : \echords(M)=0\}$.
\end{enumerate}
\end{theorem}

For a permutation $\pi$, let $\adjdes(\pi)$ be the number of descents
in $\pi$ whose letters differ by one in size.  For instance,
$\adjdes(2543176)=3$, accounted for by 54, 43 and 76.

\begin{corollary}
Given $x \in \Aseqs{}$, we have $\adjpairs(x) = \echords(\Psip(x)) =
\adjdes(\Phip(x))$.
\end{corollary}

\begin{example}\label{ex-match}
Given the ascent sequence $x=(0,1,1,0,2,0,1)$ the
corresponding permutation is $\Phip(x)= 6417325$ and the corresponding
matching is $\Psip(x)$:
$$
\scalebox{0.75}{\includegraphics{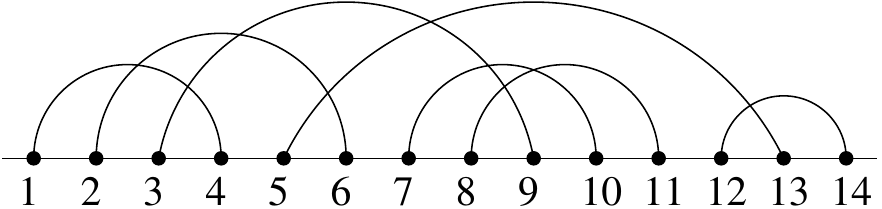}}
$$ Note that $\adjdes(6417325)=1$ because we have the two adjacent
entries $\pi_5\pi_{6}=32$.  Also $\echords(\Psip(x))=1$ since we have
one pair of similar arcs in $\Psip(x)$, namely $(7,10), (8,11)$. 
%%% REFEREE \qed
\end{example}

For a matching $M\in\Matchings_n$ let $|M|=n$ and, for $n\geq 1$, let
$A^*$ denote the arc in $M$ having the rightmost closer.  Let
$\cruns(M)$ be the number of runs of closers to the left of $A^*$.
Moreover, let $\larcs(M)$ be the number of runs of closers to the left
of the arc having the closer next to the right of A*'s opener.  For
the matching $M$ in Example~\ref{ex-match}, $|M|=7$, $A^*=(12,14)$,
$\cruns(M)=3$ (the runs of closers are 4, 6, 9(10)(11)), and
$\larcs(M)=1$ (there is one run of closers, 4, to the left of (5,13)).

Given $\pi \in \R_n$, let us label the positions of $\pi$ from left to right 
where we can insert $(n+1)$ in order to create $\pi' \in \R_{n+1}$.
Define $b(\pi)$ to be the label immediately to the left of $n$ in $\pi$.
For example, if $\pi = 6132547 \in \R_7$, then $\pi$ is labeled as
$_0 61 _1 32 _2 54 _3 7 _4$ and $b(\pi)=3$ since 3 is the label immediately
to the left of 7.

Let $\R=\bigcup_{n\ge0}\R_n$.  
Using the properties of the corresponding
bijections in \cite{BCDK} and~\cite{cdk},
\begin{eqnarray*}
G(u,v,y,t) &=& \sum_{\pi \in \R} u^{\asc(\pi^{-1})} v^{b(\pi)}
y^{\adjdes(\pi)} t^{|\pi|}\\
&=& \sum_{M \in \Matchings} u^{\cruns(M)}
v^{\larcs(M)} y^{\echords (M)} t^{|M|}.
\end{eqnarray*}
Thus, Theorem~\ref{main}
provides the generating function for the number of permutations and
matchings in question subject to 3 statistics.

Finally, as a corollary to Theorems~\ref{bcor} and~\ref{match}, we have the following enumerative result.

\begin{theorem}
$\ds \sum_{n\geq 0}|\R_{n}^{(k)}|x^n=\sum_{n\geq 0} |\Z_n^{(k)}|x^n=\sum_{n\geq 0}
\prod_{i=1}^n \left(1-\left( \dfrac{1-x}{1-x^k}
  \right)^i\right).$
\end{theorem}

\end{document}